\documentclass[11pt]{article}

\usepackage[dvips]{graphicx}
\usepackage{amsfonts,amsmath,amssymb}
\usepackage{amsmath,amsthm}
\usepackage{latexsym}
\usepackage{color}
\usepackage{tikz}

\newcommand{\pn}{{\rm pn}}
\newcommand{\cp}{\mathbin{\Box}}

\newcommand{\vertex}{\node[vertex]}
\tikzstyle{vertex}=[circle, draw, inner sep=0pt, minimum size=6pt]
\newtheorem{theorem}{Theorem}
\newtheorem{lemma}{Lemma}
\newtheorem{corollary}{Corollary}

\newtheorem{prop}{Proposition}

\begin{document}

\title{On well-dominated graphs}

\author{$^1$Sarah E. Anderson, $^2$Kirsti Kuenzel and $^3$Douglas F. Rall \\
\\
$^1$Department of Mathematics \\
University of St. Thomas \\
St. Paul, Minnesota  USA\\
\small \tt Email: ande1298@stthomas.edu  \\
\\
$^2$Department of Mathematics \\
Trinity College \\
Hartford, Connecticut  USA \\
\small \tt Email: kwashmath@gmail.com\\
\\
$^3$Department of Mathematics \\
Furman University \\
Greenville, SC, USA\\
\small \tt Email: doug.rall@furman.edu}

\date{}
\maketitle

\begin{abstract}
A graph is \emph{well-dominated} if all of its minimal dominating sets have the same cardinality.  We prove that at least one of the factors is well-dominated if the Cartesian product of two graphs is well-dominated.  In addition, we show that the Cartesian product of two connected, triangle-free graphs is well-dominated if and only if both graphs are complete graphs of order $2$.  Under the assumption that at least one of the connected graphs $G$ or $H$ has no isolatable vertices, we prove that the direct product of $G$ and $H$ is well-dominated if and only if either $G=H=K_3$ or $G=K_2$ and $H$ is either the $4$-cycle or the corona of a connected graph.  Furthermore, we show that the disjunctive product of two connected graphs is well-dominated if and only if one of the factors is a complete graph and the other factor has domination number at most $2$.
\end{abstract}


\section{Introduction}
Given a graph $G$ and a positive integer $k$, deciding whether $G$ has a dominating set of cardinality at most $k$ is one of the classic NP-complete problems~\cite{gj-1979}.  If we restrict the input graph to come from the class of well-dominated graphs (those graphs for which every minimal dominating set has the same cardinality), then this decision problem is solvable in linear time.  The following simple procedure can be used to compute the order of a minimum dominating set in a well-dominated graph $G$.  Let $D=V(G)$.  Choose any linear ordering of $V(G)$ and process the vertices one at a time in this order.  When a vertex $v$ is processed, replace $D$ by $D-\{v\}$ if $D-\{v\}$ dominates $G$.  When all vertices in the sequence have been checked the resulting set $D$ is a minimal (and hence minimum) dominating set of $G$.  Since every maximal independent set of an arbitrary graph is also a minimal dominating set, one could also compute the domination number of a well-dominated graph by using a greedy algorithm to find a maximal independent set.

The study of well-dominated graphs was initiated by Finbow, Hartnell and Nowakowski~\cite{wd}, and in that seminal paper they determined the well-dominated bipartite graphs as well as the structure of well-dominated graphs with no cycle of length less than $5$.  In~\cite{tv-1990} Topp and Volkmann gave a characterization of well-dominated block graphs and unicyclic graphs.  The $4$-connected, $4$-regular, claw-free, well-dominated graphs were characterized by Gionet, King and Sha~\cite{gks-2011}.  In more recent work, Levit and Tankus~\cite{lt-2017} proved that a graph with no cycles of length $4$ or $5$ is well-dominated if and only if it is well-covered.  G\"{o}z\"{u}pek, Hujdurovi\'{c} and Milani\v{c}~\cite{ghm-2017} determined the structure of the well-dominated lexicographic product graphs.

In this paper we first prove there are exactly eleven connected, well-dominated, triangle-free graphs whose domination number is at most $3$.  In particular, we show that $G$ is such a graph if and only if $G$ is one of $K_1,K_2,P_4,C_4,C_5,C_7$, the corona of $P_3$, or one of the four graphs in Figure~\ref{fig:dom3}.

The main focus of the paper is on well-dominated graph products, and we investigate the Cartesian, direct and disjunctive products.  The two main results concerning well-dominated Cartesian products are the following, the first of which has no restrictions on the factors while the second requires the two factors to be triangle-free.

\begin{theorem} \label{thm:wdCartesian}
Let $G$ and $H$ be connected graphs. If $G \cp H$ is well-dominated, then either $G$ or $H$ is well-dominated.
\end{theorem}

\begin{theorem} \label{thm:maincp}
Let $G$ and $H$ be nontrivial, connected graphs both of which have girth at least $4$. The Cartesian product  $G\,\Box\, H$ is well-dominated  if and only if $G = H = K_2$.
\end{theorem}

In Section~\ref{sec:DirP} we determine all well-dominated direct products if at least one of the factors
does not have a vertex that can be isolated by removing from the graph the closed neighborhood of some independent set of vertices.

\begin{theorem} \label{thm:wd-direct-no-isolatable}
Let $G$ and $H$ be nontrivial connected graphs such that at least one of $G$ or $H$ has no isolatable vertices.  The
direct product $G \times H$ is well-dominated if and only if $G=H=K_3$ or at least one of the factors is $K_2$ and the other factor is
a $4$-cycle or the corona of a connected graph.
\end{theorem}

Our main result concerning disjunctive products provides a complete characterization of well-dominated disjunctive products of two connected graphs of order at least $2$.

\begin{theorem} \label{thm:disj-characterization}
Let $G$ and $H$ be nontrivial connected graphs.  The disjunctive product $G \vee H$ is well-dominated if and only if at least one of $G$ or $H$ is a complete graph and the other factor is well-dominated with domination number at most $2$.
\end{theorem}

The paper is organized as follows.  In the next section we give the main definitions and also give some results about well-covered and well-dominated graphs that will be needed in the last three sections.  In Section~\ref{sec:smallwd} we provide a characterization of the finite list of connected triangle-free graphs that are well-dominated and have domination number less than $4$.   In Section~\ref{sec:CP} we investigate well-dominated Cartesian products and prove Theorems~\ref{thm:wdCartesian} and \ref{thm:maincp}.  Sections~\ref{sec:DirP} and~\ref{sec:DisjP} are
devoted to proving Theorem~\ref{thm:wd-direct-no-isolatable} and Theorem~\ref{thm:disj-characterization} respectively.

\section{Preliminaries}

Let $G$ be a finite, simple graph with vertex set $V(G)$ and edge set $E(G)$.  For a positive integer $n$ we always assume the vertex set of the complete graph $K_n$ is the set $[n]$, which is defined to be the set of positive integers less than or equal to $n$.  The \emph{girth} of $G$ is the length of its shortest cycle.  An edge incident with a vertex of degree $1$ is a {\it pendant} edge.  A vertex $x$ of $G$ is \emph{isolatable} if there exists an independent set $I$ in $G$ such that $x$ is an isolated vertex in $G-N[I]$.  Note that a vertex of degree $1$ that is incident with a pendant edge is isolatable in $G$ unless the pendant edge is a component of $G$.  The distance in $G$ between vertices $u$ and $v$ is the length of a shortest $u,v$-path in $G$ and is denoted $d_G(u,v)$, or $d(u,v)$ when the context is clear.

We investigate graphs that arise as a \emph{Cartesian product} $X\,\Box\, Y$, a \emph{direct product} $X \times Y$, or a \emph{disjunctive product} $X \vee Y$ of smaller (factor) graphs $X$ and $Y$.  In all three of these graph products the vertex set is $V(X) \times V(Y)$.  The edge sets are defined as follows:
\begin{enumerate}
\item $E(X \,\Box\, Y)=\{(x_1,y_1)(x_2,y_2) \,:\, (x_1=x_2 \text{ and } y_1y_2 \in E(Y)) \text{ or } (y_1=y_2 \text{ and } x_1x_2 \in E(X))\}$.
\item $E(X \times Y)=\{(x_1,y_1)(x_2,y_2) \,:\, x_1x_2\in E(X) \text{ and } y_1y_2 \in E(Y)\}$.
\item $E(X \vee Y)=\{(x_1,y_1)(x_2,y_2) \,:\, x_1x_2\in E(X) \text{ or } y_1y_2 \in E(Y)\}$.
\end{enumerate}
Suppose $x \in V(X)$ and $y \in V(Y)$.  If $\ast \in \{\,\Box\,, \times, \vee\}$, then the \emph{$X$-layer} of $X \ast Y$ containing $(x,y)$ is the subgraph of $X \ast Y$ induced by the set $V(X) \times \{y\}$ and the \emph{$Y$-layer} containing $(x,y)$ is the subgraph induced by $\{x\} \times V(Y)$.  Note that in the Cartesian product or the disjunctive product, an $X$-layer is isomorphic to $X$ while in the direct product an $X$-layer is a totally disconnected graph of order $|V(X)|$.  A similar statement holds for a $Y$-layer.

For $x\in V(G)$ the \emph{open neighborhood} of $x$ is the set $N(x)$ consisting of all vertices in $G$ that are adjacent to $x$ and its \emph{closed neighborhood} is the set defined by $N[x]=N(x)\cup \{x\}$.  Let $S \subseteq V(G)$.  The open (closed) neighborhood of $S$ is the set $N(S)$ ($N[S]$) defined as the union of the open (closed) neighborhoods of the vertices in $S$.  For $v \in S$, the \emph{closed private neighborhood of $v$ with respect to $S$} is the set $\pn[v, S]$ defined by $\pn[v, S] = \{u \in V(G) : N[u] \cap S = \{v\}\}$.  If $\pn[v,S]$ is nonempty, then each vertex in $\pn[v,S]$ is called a \emph{private neighbor} of $v$ with respect to $S$.  A set $D \subseteq V(G)$ is a \emph{dominating set} of $G$ if $N[D]=V(G)$, and we then say that $D$ \emph{dominates} $G$.  If $\{x\}$ dominates $G$, then $x$ is called a \emph{universal} vertex.  A dominating set $D$ is minimal with respect to set inclusion if and only if $\pn[u,D]\neq \emptyset$ for every $u \in D$.  The smallest cardinality among the minimal dominating sets is called the \emph{domination number} of $G$ and is denoted $\gamma(G)$.  The \emph{upper domination number} of $G$ is the number $\Gamma(G)$, which is the largest cardinality of a minimal dominating set.  If $N(S)=V(G)$, then $S$ is a \emph{total dominating} set.  The smallest and largest cardinalities of  minimal total dominating sets in $G$ are the \emph{total domination number} $\gamma_t(G)$ (respectively, the \emph{upper total domination number} $\Gamma_t(G)$) of $G$.

As defined by Finbow et al.~\cite{wd}, the graph $G$ is \emph{well-dominated} if every minimal dominating set of $G$ has the same cardinality.  That is, $G$ is well-dominated if and only if $\gamma(G)=\Gamma(G)$. The cardinality of the smallest maximal independent set in $G$ is the \emph{independent domination number} of $G$ and is denoted $i(G)$.  The \emph{vertex independence number}, $\alpha(G)$, is the cardinality of a largest independent set of vertices in $G$.  A graph $G$ was defined by Plummer~\cite{p-1970} to be \emph{well-covered} if every maximal independent set of $G$ has cardinality $\alpha(G)$.  Since any maximal independent set in $G$ is also a minimal dominating set, we get the well-known inequality chain
\[\gamma(G) \le i(G) \le \alpha(G) \le \Gamma(G)\,,\]
which immediately gives the following result.

\begin{prop} {\rm \cite[Lemma 1]{wd}} \label{prop:wd-Implies-wc}
Every well-dominated graph is well-covered.
\end{prop}

The {\it corona} of  $G$ is denoted by $G \circ K_1$ and is formed by adding a single (new) vertex of degree $1$ adjacent to each vertex of $G$.  It is easy to show that the corona of any graph is well-dominated.  Furthermore, each complete graph is well-dominated and $P_1,P_2,P_4,C_3,C_4,C_5,C_7$ is a complete list of the well-dominated paths and cycles.

The following result follows directly from the definitions.
\begin{lemma} \label{lem:useful}
If $G$ is a well-covered graph and $I$ is any maximal independent set in $G$, then the subgraph of $G$ induced by
$\pn[x,I]$ is a complete subgraph for every $x \in I$.
\end{lemma}

In \cite{girth5}, Finbow et al.  classified all connected well-covered graphs of girth at least $5$. In doing so they
defined a class of graphs, denoted $\mathcal{PC}$, as follows. A $5$-cycle $C$ of a graph $G$ is called \textit{basic}
if $C$ does not contain two adjacent vertices of degree three or more in $G$. A graph $G$ belongs to $\mathcal{PC}$ if $V(G)$
can be partitioned into two subsets $P$ and $C$ where
\begin{itemize}
\item $P$ contains the vertices incident with a pendant edge, and the pendant edges form a perfect matching of $P$; and
\item $C$ contains the vertices of the basic $5$-cycles, and the basic $5$-cycles form a partition of $C$.
\end{itemize}

A well-covered graph $G$ in $\mathcal{PC}$ of girth at least $5$ need not have any basic $5$-cycles.  In this case it is clear that $G$ is the corona of the graph $H$ obtained by deleting all the vertices of degree $1$ from $G$.

\begin{theorem}{\rm \cite{girth5}} \label{thm:girth5}
Let $G$ be a connected well-covered graph of girth at least $5$. Then $G$ is in $\mathcal{PC}$ or $G$ is isomorphic to one
of $K_1$, $C_7$, $P_{10}$, $P_{13}$, $Q_{13}$, or $P_{14}$.
\end{theorem}

The well-covered graph $P_{10}$ is shown in Figure~\ref{fig:twoorphans1}.  The other three special well-covered graphs,
$P_{13}$, $P_{14}$ and $Q_{13}$, from Theorem~\ref{thm:girth5} are not well-dominated and hence will not concern us here.

\begin{figure}[ht!]
\tikzstyle{every node}=[circle, draw, fill=black!0, inner sep=0pt,minimum width=.2cm]
\begin{center}
\begin{tikzpicture}[thick,scale=.6]
  \draw(0,0) { % <-- START CO-ORDINATES
    %%%%EDGES
    +(1.33,4.67) -- +(0.00,3.33)
    +(2.67,3.33) -- +(1.33,4.67)
    +(0.00,3.33) -- +(0.00,1.33)
    +(0.00,1.33) -- +(2.67,1.33)
    +(2.67,1.33) -- +(2.67,3.33)
    +(2.67,3.33) -- +(5.33,3.33)
    +(5.33,3.33) -- +(5.33,1.33)
    +(5.33,1.33) -- +(4.00,0.00)
    +(4.00,0.00) -- +(2.67,1.33)
    +(1.33,4.67) -- +(6.67,4.67)
    +(6.67,4.67) -- +(6.67,0.00)
    +(6.67,0.00) -- +(4.00,0.00)
    %%%%VERTICES
    +(1.33,4.67) node{}
    +(0.00,3.33) node{}
    +(2.67,3.33) node{}
    +(0.00,1.33) node{}
    +(2.67,1.33) node{}
    +(5.33,1.33) node{}
    +(5.33,3.33) node{}
    +(4.00,0.00) node{}
    +(6.67,4.67) node{}
    +(6.67,0.00) node{}
    %%%%LABELS
    };
\end{tikzpicture}

\end{center}
\caption{$P_{10}$}
\label{fig:twoorphans1}
\end{figure}

In \cite{wd} they show precisely which of those well-covered graphs of girth at least $5$ are also well-dominated.

\begin{theorem}{\rm \cite{wd}} \label{thm:wdgirthatleast5}
Let $G$ be a connected, well-dominated graph of girth at least $5$. Then $G \in \mathcal{PC}$ if and only if for every
pair of basic $5$-cycles there is either no edge joining them, exactly two vertex-disjoint edges joining them, or exactly
four edges joining them. If $G \not\in \mathcal{PC}$, then $G$ is isomorphic to $K_1$, $C_7$, or $P_{10}$.
\end{theorem}

\section{Triangle-free well-dominated graphs} \label{sec:smallwd}

In this section we determine the finite set of connected, well-dominated, triangle-free graphs whose domination number is at most $3$.
It is clear that if $G$ is well-dominated with domination number $1$, then $G$ is a complete graph.  Therefore, we begin with the study of graphs with domination number $2$.

\begin{theorem} \label{thm:dom2girth4} If $G$ is a connected, well-dominated graph of girth at least $4$ and domination
number $2$, then $G \in \{P_4, C_4 ,C_5\}$.
\end{theorem}

\begin{proof}  Suppose $G$ is a connected, well-dominated graph of girth at least $4$ such that $\gamma(G)=2$.
Let $S = \{u,v\}$ be an independent set in $G$.  Since $G$ is well-covered by Proposition~\ref{prop:wd-Implies-wc}, $S$ is also a minimum dominating set of $G$ and thus $N[S]=V(G)$.  Moreover, both $\pn[u, S]$ and $\pn[v, S]$ induce a clique in $G$ by Lemma~\ref{lem:useful}.  Since $G$ is triangle-free,
it follows that $|\pn[u, S]| \le 2$ and $|\pn[v, S]| \le 2$.

Suppose first that $\pn[v, S] = \{v\}$ and $\pn[u, S] = \{u\}$. Since $G$ is connected, $N(u)=N(v) \ne \emptyset$.
Since $G$ is triangle-free, $N(u)$ is independent. Thus, $|N(u)| \le \alpha(G) = 2$. However, if $|N(u)| = 1$, then $G = P_3$,
which is not a well-dominated graph. Therefore, $|N(u)|= 2$ and $G = C_4$.

Next, suppose that $\pn[v, S] = \{v\}$ and $\pn[u, S] = \{u, w\}$. Since $G$ is connected and triangle-free, $N(u)$
and $N(v)$ are  independent sets and $N(u)\cap N(v) \ne \emptyset$.  Thus $|N(u) \cap N(v)| \le \alpha(G) = 2$.  If
$|N(u) \cap N(v)| = 1$, then $G = P_4$. However, if $|N(u) \cap N(v)| = 2$, then $N(u) \cup N(v)$ is an independent set
of size $3$, which is a contradiction.

Finally, suppose that $\pn[v, S] = \{v, z\}$ and $\pn[u, S] = \{u, w\}$. If $N(u) \cap N(v) = \emptyset$, then
$wz \in E(G)$ and $G = P_4$. If  $N(u) \cap N(v) \ne \emptyset$, then $N(u) \cup N(v)$ is an independent set of size
at least $3$ unless $zw \in E(G)$. In this case, $G = C_5$.
\end{proof}

Next we  classify all connected, well-dominated graphs with domination number $3$ and girth at least $4$.  For this purpose let
$\mathcal{F}_1$ be the set of four graphs $H_1,H_2,H_3$ and $H_4$ depicted in Figure~\ref{fig:dom3}.

 \begin{theorem} \label{thm:dom3girth4}
 If $G$ is a connected well-dominated graph of girth at least $4$ and domination number $3$, then $G\in \{P_3 \circ K_1, C_7\}$ or
 $G \in \mathcal{F}_1$.
 \end{theorem}

 \begin{proof}
 Let $G$ be a connected, well-dominated graph of girth at least $4$ such that $\gamma(G)=3$.  This implies that $G$ is also well-covered
 with $\alpha(G)=3$, which in turn implies that $\Delta(G) \le 3$.  Suppose first that the girth of $G$ is at least $5$.  Since $\gamma(P_{10})=4$ and $\gamma(G)=3$, it  follows by Theorem~\ref{thm:wdgirthatleast5} that $G \in \mathcal{PC}$ or $G = C_7$.
 If $G \in \mathcal{PC}$, then $G$ contains either one basic $5$-cycle and one pendant edge or $G$ contains three pendant edges.  Therefore,  $G \in\{P_3 \circ K_1, C_7, H_1\}$, and the conclusion holds.  Hence, we shall assume $G$ contains a $4$-cycle $u_1u_2u_3u_4u_1$.

 Suppose first that there exists a vertex $x$ such that $d(x, u_i) \ge 2$ for each $1 \le i \le 4$. Suppose that $G- N[x]$ contains at least $5$ vertices and let $z \in V(G) - (N[x] \cup \{u_1, u_2, u_3, u_4\})$. If $z$ is adjacent to some vertex in $\{u_1, u_2, u_3, u_4\}$, then we may assume that $z$ is adjacent to $u_1$. In any case,  $\{z, u_2, u_4, x\}$ is an  independent set of size $4$, which is a contradiction. Therefore, we may assume $G - N[x] = u_1u_2u_3u_4u_1$.   Since $\Delta(G) \le 3$,  the vertex $x$ has degree at most $3$. Suppose first that $\deg(x) = 1$ and $N(x) = \{w\}$. Since
 $G$ is connected, we may assume with no loss of generality that $wu_3 \in E(G)$. Note that $N(w) = \{x, u_1, u_3\}$ or $N(w) = \{x, u_3\}$.
 In either case, $\gamma(G) = 2$,  which is a contradiction.

 Next, suppose $\deg(x) = 2$ and let $N(x) = \{w_1, w_2\}$. We may assume that $w_1u_4 \in E(G)$. Notice that $\{w_1, w_2, u_1, u_3\}$ is not an independent set.   Therefore, $w_2u_1$ or $w_2u_3$ is an edge in $G$. With no loss of generality, we may assume $w_2u_3 \in E(G)$. If we have identified all the edges of $G$, then  $G = H_3$. If $w_2u_1$ is the only other edge in $G$, then $G = H_4$. If $w_2u_1\in E(G)$, then the only other edge that can be in $G$ is $w_1u_2$.  However, in  this case, $\gamma(G) = 2$, which is a contradiction. So we shall assume that $w_2u_3 \in E(G)$ and $w_2u_1 \not\in E(G)$. The only other  additional edge that $G$  may contain is $w_1u_2$ which creates a graph that is isomorphic to $H_4$.

 Finally, suppose $\deg(x) = 3$ and let $N(x) = \{w_1, w_2, w_3\}$. Without loss of generality, we may assume $w_1u_4 \in E(G)$. Since $G$ is well-covered,  $\deg(w_2) \ge 2$ or $\deg(w_3) \ge 2$ since no well-covered graph contains a vertex with more than one neighbor of degree $1$. We may assume that $\deg(w_2) \ge 2$. Note that assuming  $w_2u_3 \in E(G)$ is equivalent to assuming that $w_2u_1 \in E(G)$. Therefore, we consider two possibilities: $w_2u_3 \in E(G)$ or $w_2u_2 \in E(G)$.

 Suppose first that $w_2u_3 \in E(G)$. Note that $\{u_1, w_1, w_2, w_3\}$ is not an independent set and $\{u_2, w_1, w_2, w_3\}$ is not an independent set.  Therefore, either $u_1w_2$ or $u_1w_3$ is in $E(G)$, and either $u_2w_1$ or $u_2w_3$ is in $E(G)$. Suppose first that $u_1w_2$ and $u_2w_1$ are edges in $G$.  Note that all vertices other than $w_3$ have degree $3$ so $G$ contains no other edges. However, it now follows that $\{u_1, u_3, w_1, w_3\}$ is an independent set, which is a  contradiction. Therefore, this case cannot occur. If $u_1w_3$ and $u_2w_1$ are in $E(G)$, then $\{u_2, u_4, w_2, w_3\}$ is an independent set in $G$, another  contradiction. We conclude that $u_2w_1 \not\in E(G)$, and so $u_2w_3 \in E(G)$, which in turn implies that $u_1w_2 \in E(G)$. However, in this case,
 $\{u_1, u_3, w_1, w_3\}$ is an independent set in $G$. Thus, it must be that $w_2u_3 \not\in E(G)$.

 Therefore, we may assume that $w_2u_2 \in E(G)$. Since $\{u_1, u_3, w_1, w_2, w_3\}$ is not an independent set, it follows that $w_3u_1$ and $w_3u_3$ are  edges in $G$. However, $\{u_1, u_3, w_2, w_1\}$ is now an independent set in $G$, another contradiction.

 Having exhausted all possibilities when $G$ contains a vertex $x$ that is distance $2$ from $u_1u_2u_3u_4u_1$, we now consider when every vertex of  $V(G) - \{u_1, u_2, u_3, u_4\}$ is adjacent to $u_i$ for some $i \in [4]$. Let $V_i $ be the set of vertices in $V(G) - \{u_1, u_2, u_3, u_4\}$ adjacent  to $u_i$. Since $\deg(v) \le 3$ for each $v \in V(G)$, $|V_i| \le 1$. We may assume that $|V_1| = 1$ and we write $V_1 = \{v_1\}$. If $|V_2| = 0 = |V_4|$,  then $\gamma(G) = 2$. So we may assume $V_2 = \{v_2\}$. Note that $v_1 \ne v_2$. Moreover, if $|V_3| = 0 =|V_4|$, then $\gamma(G) = 2$. So we may assume  $V_4 = \{v_4\}$ and we know that $v_4 \ne v_1$. However, it may be the case that $v_4 = v_2$. Suppose first that $v_4 \ne v_2$. Since $\{v_2, v_4, u_1, u_3\}$  is not independent and $G$ is triangle-free, it follows that $v_2v_4 \in E(G)$.  Suppose first that  $V_3 = \emptyset$. If we have identified all the edges in $G$, then $G$ is isomorphic  to $H_2$. If $G$ contains the edge $v_1v_4$, then $\{v_4, u_2\}$ dominates $G$. On the other hand, if  $v_1v_2 \in E(G)$, then $\{v_2, u_4\}$ dominates $G$.
 So we may assume that $V_3 \ne \emptyset.$ Suppose first that $V_3 = \{v_1\}$. Thus, $v_1u_3 \in E(G)$. If we have identified all the edges in $G$, then $G = H_4$.   Otherwise, $G$ also contains $v_1v_2$ or $v_1v_4$. However, if $G$ contains the edge $v_1v_2$, then $\{v_2, u_4\}$ dominates $G$, which is a contradiction.
 Similarly,  $v_1v_4 \not\in E(G)$. So we shall assume that $V_3 = \{v_3\}$ where $v_3 \ne v_1$. Note that $v_1v_3 \in E(G)$ for otherwise
 $\{v_1, v_3, u_2, u_4\}$ is an independent set. If these are the only edges in $G$, then $\{v_1, v_2, v_3, v_4\}$ is a minimal dominating set which is a  contradiction. Therefore, $G$ must contain one of the edges $v_1u_3$, $v_2u_4$, $v_3u_1$, or $v_4u_2$. However, the addition of any one of these edges results in a vertex with degree $4$, which is a contradiction.

 Having exhausted all possibilities for when $v_2 \neq v_4$, we finally consider the case when $v_2 = v_4$. This implies that $V_3 = \{v_3\}$  where $v_3 \ne v_1$, for otherwise $\{u_1, u_4\}$ dominates $G$. However, this case is equivalent to the case where
 $V_i = \{v_i\}$   for $i \in \{1, 2, 4\}$  and $V_3 = \{v_1\}$. Hence, we have identified all connected, well-dominated graphs with girth at least $4$ and domination number $3$.
 \end{proof}

\begin{figure}[ht!]
\begin{center}
\begin{tikzpicture}[auto]
    % Place nodes

    	\vertex(-1) at (-5,0) [label=above:$$,scale=.75pt,fill=black]{};
	\vertex(-2) at (-5.75, -1)[label=left:$$,scale=.75pt,fill=black]{};
	\vertex(-3) at (-4.25, -1)[label=right:$$,scale=.75pt,fill=black]{};
	\vertex(-4) at (-3, -1)[label=right:$$,scale=.75pt,fill=black]{};
	\vertex(-5) at (-5.75, -2)[label=left:$$,scale=.75pt,fill=black]{};
	\vertex(-6) at (-4.25, -2)[label=below:$$,scale=.75pt,fill=black]{};
	\vertex(-7) at (-3, -2)[label=right:$$,scale=.75pt,fill=black]{};

	\vertex (1) at (.5,0) [label=above:$$,scale=.75pt,fill=black]{};
	\vertex (2) at (-.25, -1) [label=left:$$,scale=.75pt,fill=black]{};
	\vertex (3) at (1.25, -1) [label=right:$$,scale=.75pt,fill=black]{};
	\vertex (4) at (2.5,-1) [label=right:$$,scale=.75pt,fill=black]{};
	\vertex (5) at (-.25, -2) [label=left:$$,scale=.75pt,fill=black]{};
	\vertex (6) at (1.25,-2) [label=below:$$,scale=.75pt,fill=black]{};
	\vertex (7) at (2.5, -2) [label=right:$$,scale=.75pt,fill=black]{};

	\vertex (21) at (-5,-4) [label=above:$$,scale=.75pt,fill=black]{};
	\vertex (22) at (-5.75,-5) [label=above:$$,scale=.75pt,fill=black]{};
	\vertex (23) at (-4.25, -5) [label=above:$$,scale=.75pt,fill=black]{};
	\vertex (24) at (-3,-5) [label=above:$$,scale=.75pt,fill=black]{};
	\vertex (25) at (-5.75, -6) [label=right:$$,scale=.75pt,fill=black]{};
	\vertex (26) at (-4.25,-6) [label=below:$$,scale=.75pt,fill=black]{};
	\vertex (27) at (-3, -6) [label=left:$$,scale=.75pt,fill=black]{};

	\vertex (2A) at (.5,-4) [label=above:$$,scale=.75pt,fill=black]{};
	\vertex (2B) at (-.25,-5) [label=above:$$,scale=.75pt,fill=black]{};
	\vertex (2C) at (1.25, -5) [label=above:$$,scale=.75pt,fill=black]{};
	\vertex (2D) at (2.5,-5) [label=above:$$,scale=.75pt,fill=black]{};
	\vertex (2E) at (-.25, -6) [label=right:$$,scale=.75pt,fill=black]{};
	\vertex (2F) at (1.25,-6) [label=below:$$,scale=.75pt,fill=black]{};
	\vertex (2G) at (2.5, -6) [label=left:$$,scale=.75pt,fill=black]{};

	\node(A) at (-5, -3)[]{(a) $H_1$};
	\node(B) at (1.5, -3)[]{(b) $H_2$};
	\node(E) at (-5, -7)[]{(c) $H_3$};
	\node(G) at (1.5, -7)[]{(d) $H_4$};

	\path
		(1) edge (2)
		(1) edge (3)
		(2) edge (5)
		(3) edge (6)
		(5) edge (6)
		(4) edge (7)
		(7) edge (6)
		(7) edge[bend right=30] (1)
		(21) edge (22)
		(21) edge (23)
		(22) edge (25)
		(23) edge (26)
		(25) edge (26)
		(24) edge (27)
		(27) edge (26)
		(23) edge (24)

		(2A) edge (2B)
		(2A) edge (2C)
		(2B) edge (2E)
		(2C) edge (2F)
		(2E) edge (2F)
		(2D) edge (2G)
		(2G) edge (2F)
		(2D) edge (2C)
		(2G) edge[bend right=30] (2A)

		(-1) edge (-2)
		(-1) edge (-3)
		(-2) edge (-5)
		(-3) edge (-6)
		(-5) edge (-6)
		(-4) edge (-7)
		(-7) edge (-6)

		;

\end{tikzpicture}
\end{center}
\caption{The class $\mathcal{F}_1$}
\label{fig:dom3}
\end{figure}
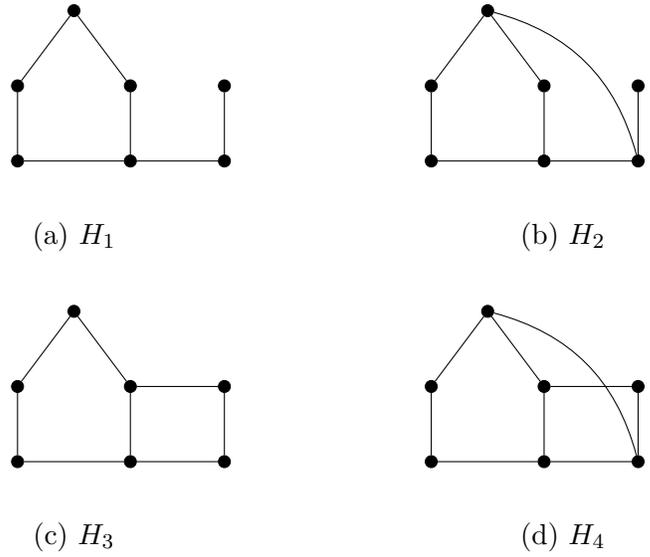

Combining Theorem~\ref{thm:dom2girth4} and Theorem~\ref{thm:dom3girth4} we have shown that a connected, triangle-free graph $G$ such that $\gamma(G) \le 3$ is well-dominated if and only if  $G$ is one of $K_1,K_2,P_4,C_4,C_5,C_7$, the corona of $P_3$, or $G \in \mathcal{F}_1$.

\section{Cartesian Products} \label{sec:CP}
In 2013 Hartnell and Rall proved that if a Cartesian product is well-covered then at least one of the factors is well-covered.
(See~\cite[Theorem 2]{ha-ra-2013}.)  In this section we prove a corresponding result for well-dominated Cartesian products.
Furthermore, we give a complete characterization of triangle-free, well-dominated Cartesian products.
We will need the following concept.   A set $S \subseteq V(G)$ is {\it open irredundant} if $N(u)-N[S-\{u\}] \neq \emptyset$
for every vertex $u\in S$.  That is, $S$ is open irredundant if every vertex of $S$ has a private neighbor (with respect to $S$) that
belongs to $V(G)-S$.  The following result of Bollob\'{a}s and Cockayne will prove useful in the proof of Theorem~\ref{thm:wdCartesian}.

\begin{lemma} {\rm \cite{bc-1992}}
If a graph $G$ has no isolated vertices, then $G$ has a minimum dominating set that is open irredundant.
\end{lemma}

\medskip

\noindent \textbf{Theorem~\ref{thm:wdCartesian}} \emph{
Let $G$ and $H$ be connected graphs. If $G\,\Box\, H$ is well-dominated, then either $G$ or $H$ is well-dominated.
}
\begin{proof}
Suppose $G$ and $H$ are connected graphs and that $G\,\Box\, H$ is well-dominated.  This implies that $G\,\Box\, H$ is well-covered, and by~\cite{ha-ra-2013} either $G$ or $H$ is well-covered. Without loss of generality, assume that $G$ is
well-covered. Note that if $G$ is also well-dominated, then we are done. So we shall assume that $G$ is not well-dominated. Choose a
minimum dominating set $I$ of $G$ such that $I$ is open irredundant. The set $I\times V(H)$ is a dominating set of $G\,\Box\, H$.
Since $I$ is open irredundant in $G$, every vertex of $I \times V(H)$ has a private neighbor in its $G$-layer.  Therefore $I \times V(H)$ is
a minimal dominating set of $G \,\Box\, H$ and hence is a minimum dominating set since $G \,\Box\, H$ is well-dominated.
Next, let $D_G$ be any minimal dominating set of $G$ and let $A=\{ x \in D_G \,\colon\, {\rm pn}[x,D_G]=\{x\}\}$.
We create a minimal dominating set of $G\,\Box\, H$ as follows. Choose a minimal dominating set $D_H$ of $H$ and let $S=(A\times D_H) \cup ((D_G -  A) \times V(H))$.
We claim that $S$ is a minimal dominating set of $G\,\Box\, H$. To see this, note first that $S$ dominates $G\,\Box\, H$.   Furthermore, every vertex in
$D_G -  A$ has a private neighbor (with respect to $D_G$) in $V(G) -  D_G$. Thus, every vertex $(g, h) \in (D_G-A) \times V(H)$ has a
private neighbor with respect to $S$ in its $G$-layer. Next, let $(g,h)\in A\times D_H$.  Since $D_H$ is a minimal dominating set of $H$,
it follows that $h$ has a private neighbor (possibly itself) with respect to $D_H$ in $H$.   Hence, $(g, h)$ has a private neighbor with respect to $S$  in its $H$-layer. Consequently, $S$ is a minimal dominating set of $G\,\Box\, H$, and therefore $|S| =|I||V(H)|$ since $G \,\Box\, H$ is well-dominated. Furthermore,  for any two minimal dominating sets $D_1$ and $D_2$ of $H$,
\[|A\times D_1| + |(D_G -  A)\times V(H)| = |I||V(H)|=|A\times D_2| + |(D_G -  A)\times V(H)|\,.\]
This implies that $|A\times D_1|=|A\times D_2|$.  Thus, either $H$ is well-dominated (and the theorem is proved) or $A = \emptyset$.
If $A = \emptyset$, then the above equation becomes $|D_G||V(H)| = |I||V(H)|$, which implies $|D_G|=|I|=\gamma(G)$.
It follows that $G$ is well-dominated.
\end{proof}

Hartnell et al. proved the following theorem concerning well-covered Cartesian products of graphs having no triangles.

\begin{theorem} {\rm \cite{wellcoverCart}}\label{thm:girth4cart}
If $G$ and $H$ are nontrivial, connected graphs with girth at least $4$ such that $G\,\Box\, H$ is well-covered, then at least one of $G$ or $H$ is the graph $K_2$.
\end{theorem}

\begin{lemma} \label{lem:wdprisms}
Suppose $G$ is a nontrivial, connected graph.  If the Cartesian product $G \,\Box\, K_2$ is well-dominated, then $G =K_2$.
\end{lemma}
\begin{proof}
Suppose that $G$ is a nontrivial, connected graph such that $G \,\Box\, K_2$ is well-dominated. Note that $\{(u,1): u \in V(G)\}$ is a minimal dominating
set of $G \,\Box\, K_2$. Thus, $\gamma(G\,\Box\, K_2) = |V(G)|$.  Suppose $G$ contains a vertex, $w$, of degree at least $2$. Choose any
$x\in N(w)$. We claim that $D = \{(u,1): u \not\in \{x, w\}\} \cup \{(x, 2)\}$  is a dominating set of $G \,\Box\, K_2$. To see this, note that $(u, 2)$ is dominated
by $(u,1)$ for all $u \not\in \{w, x\}$, while $(x, 2), (x, 1)$, and $(w, 2)$  are dominated by $(x, 2)$. Moreover, there exists
$z \in N(w) - \{x\}$ such that $(z, 1) \in D$, and $(z,1)$ dominates $(w,1)$.  This is a contradiction since $|D| \le |V(G)|-1$.
We conclude that $\Delta(G)=1$, which implies that $G = K_2$.
\end{proof}

We now proceed with the proof of our main result of the section that characterizes connected, well-dominated Cartesian products that are triangle-free.  For the sake of convenience we restate it here.

\medskip
\noindent \textbf{Theorem~\ref{thm:maincp}} \emph{
Let $G$ and $H$ be nontrivial, connected graphs both of which have girth at least $4$. The Cartesian product  $G\,\Box\, H$ is well-dominated  if and only if $G = H = K_2$.
}
\begin{proof}
Suppose $G$ and $H$ are nontrivial, connected graphs both of which have girth at least $4$ such that $G\,\Box\, H$ is well-dominated.  By Proposition~\ref{prop:wd-Implies-wc}, $G\,\Box\, H$ is well-covered. Combining Theorem~\ref{thm:girth4cart} and Lemma~\ref{lem:wdprisms} it follows that $G = H =  K_2$.  Since $K_2 \,\Box\, K_2=C_4$, the converse is clear.
\end{proof}

 \section{Direct Products} \label{sec:DirP}

In this section we investigate direct products of two connected graphs such that at least one of them does not have any isolatable vertices.  We first list some known results about domination of direct products and about direct products that are well-covered.

\begin{theorem}{\rm \cite{domdirect}} \label{thm:upperbound}
For any graphs $G$ and $H$, $\gamma(G\times H) \le 3\gamma(G)\gamma(H)$.
\end{theorem}

\begin{theorem} {\rm \cite{KKDFR}}\label{thm:directK_n}
Let $G$ and $H$ be nontrivial, connected graphs such that the direct product $G\times H$ is well-covered. If $H$ has no isolatable vertices,
then $H$ is a complete graph.
\end{theorem}

The domination number of a graph with no isolated vertices is at most one-half its order.  The following result further restricts the relative size of a minimum dominating set of the factors of a well-dominated direct product.

\begin{lemma}\label{lem:3gamma}
Suppose $G$ and $H$ are graphs without isolated vertices.   If $G\times H$ is well-dominated, then $3\gamma(G) \ge |V(G)|$ and $3\gamma(H) \ge |V(H)|$.
\end{lemma}

\begin{proof}
Suppose that both $G$ and $H$ have no isolated vertices and that $G \times H$ is well-dominated.  Let $I$ be a maximum independent set of $G$. It follows that $I\times V(H)$ is a maximal independent set, and thus also a minimum dominating set of $G\times H$.  By Theorem~\ref{thm:upperbound} it follows that
\[3\gamma(G)\gamma(H) \ge \gamma(G \times H) = |I| |V(H)| = \alpha(G)|V(H)| \ge \gamma(G)|V(H)|.\]
Therefore, $3\gamma(H) \ge |V(H)|$. Similarly, $3\gamma(G) \ge |V(G)|$.
\end{proof}

\begin{corollary}
If $H$ has no isolatable vertices and $G$ is any nontrivial graph such that $G\times H$ is well-dominated, then $H \in \{K_2, K_3\}$.
\end{corollary}

\begin{proof}
By Proposition~\ref{prop:wd-Implies-wc}, $G \times H$ is well-covered, and it follows by
Theorem~\ref{thm:directK_n}  that $H$ is a complete graph. By Lemma~\ref{lem:3gamma}, $3=3\gamma(H)  \ge |V(H)|$,
and thus $H \in \{K_2, K_3\}$.
\end{proof}

We need the following theorem of Topp and Volkmann concerning well-covered direct products and a characterization by  Payan and Xuong of graphs whose domination number is one-half their order.

\begin{theorem} {\rm \cite{topp}} \label{thm:2Properties}
If $G$ and $H$ are graphs without isolated vertices and $G\times H$ is well-covered, then
\begin{enumerate}
\item[1.] $G$ and $H$ are well-covered and
\item[2.] $\alpha(G)|V(H)| = \alpha(H)|V(G)|$.
\end{enumerate}
\end{theorem}

\begin{theorem} {\rm \cite{pa-xu-1982}} \label{thm:one-half-order}
If $G$ is a connected graph of order $n\ge 2$, then $\gamma(G)=n/2$ if and only if $G=C_4$ or $G=H\circ K_1$ for some connected graph $H$.
\end{theorem}

We are now able to characterize those well-dominated direct products in which at least one of the factors is $K_2$.

\begin{lemma} \label{lem:onefactork2}
Let $G$ be a nontrivial connected graph. The direct product $G\times K_2$ is well-dominated if and only if $G=C_4$
or $G$ is the corona of a connected graph.
\end{lemma}
\begin{proof}
Let $G$ be a nontrivial connected graph.  Suppose first that  $G\times K_2$ is well-dominated.
Let $D$ be an arbitrary minimal dominating set of $G$ and let $S=D\times V(K_2)$.  From the definition of direct product it is clear that $S$
dominates $G \times K_2$.  We claim that $S$ is a minimal dominating set. To see this, without loss of generality consider  $(x, 1) \in S$.
Since $D$ is a minimal dominating set of $G$, the vertex  $x$ has a private neighbor, say $x'$, with respect to $D$. That is, $x'$ is a
vertex of $G$ such that $N[x']\cap D=\{x\}$.  If $x' = x$, then $(x, 1)$ is its own private neighbor with respect to $S$. On the other hand, if $x' \ne x$, then $(x',2)$ is a private neighbor of $(x, 1)$ with respect to $S$. This proves that $S$ is a minimal dominating set of $G \times K_2$. Now, if $D_1$ and $D_2$ are two minimal dominating sets of $G$, then  $|D_1 \times V(K_2)| = |D_2 \times V(K_2)|$ since $G \times K_2$
is well-dominated.  Therefore, $G$ is well-dominated.  Furthermore, since $G \times K_2$ is well-dominated, it is well-covered and hence
by Theorem~\ref{thm:2Properties},  $\gamma(G)=\alpha(G)=\frac{1}{2}|V(G)|$.  It now follows from Theorem~\ref{thm:one-half-order} that $G$ is either a $4$-cycle or the corona of a connected graph.

For the converse suppose that $G=F \circ K_1$, for some connected graph $F$.  Let $V(F)=\{x_1,\ldots,x_n\}$ and for each $i\in [n]$ let $y_i$
be the vertex of degree $1$ adjacent in $G$ to $x_i$.  By the definition of direct product, the graph $G \times K_2$ is a graph in which each vertex in
$\{y_1, \ldots, y_n\} \times [2]$ has degree $1$ and each vertex in $\{x_1,\ldots,x_n\} \times [2]$ is adjacent to a single vertex of degree $1$.
That is, $G \times K_2$ is itself a corona and is therefore well-dominated.  Also, $C_4 \times K_2=2C_4$, which is well-dominated.
\end{proof}

Next, we consider well-dominated products of the form $G\times K_3$.  Recall that a subset of vertices in a graph is a {\it $2$-packing} if the distance between any pair of distinct vertices in the set is at least $3$.

\begin{lemma}\label{lem:2packing}
If $G$ is a connected graph such that $G\times K_3$ is well-dominated, then every maximal independent set in $G$ is in fact a $2$-packing.
\end{lemma}

\begin{proof}
Suppose $G$ is connected, $G\times K_3$ is well-dominated and $I$ is a maximal independent set in $G$. If every vertex in $V(G) - I$ is adjacent to only one vertex of $I$, then $I$ is a $2$-packing. So we may assume that there exists $w \in V(G) - I$ such that $w$ is adjacent to at least two vertices in $I$. Let $Z = N(w) \cap I$ and choose a minimum subset $Z_1$ in $Z$ that dominates $N(Z)-(\{w\} \cup N(I-Z))$.
 We claim that  \[D = \left((I-Z) \times \{1, 2, 3\}\right) \cup \left( Z_1 \times \{1, 2\}\right) \cup \left(((Z- Z_1) \cup \{w\}) \times \{1\}\right)\]
is a minimal dominating set of $G\times K_3$. First, we show that $D$ does in fact dominate $G\times K_3$. Let $(u,v) \in V(G\times K_3)- D$. Thus, $u \not\in I - Z$. If $u \in Z_1$ and $v = 3$, then $(w, 1)$ dominates $(u, v)$. Similarly, if $u \in Z- Z_1$, then $ v \in \{2,3\}$ and
$(w, 1)$ dominates $(u,v)$. Therefore, we shall assume that $u \in V(G) - I$. If $u = w$ and $v \in \{2, 3\}$, then $(x, 1)$ dominates $(u,v)$ for any $x \in Z$. If $u \in V(G) - (I \cup \{w\})$, then for some $x \in (I - Z) \cup Z_1$ the set $\{(x, 1), (x, 2)\}$ dominates $(u,v)$. Thus, $D$ dominates $G \times K_3$.

Next, we show that $D$ is a minimal dominating set of $G\times K_3$. The set $D - \{(w, 1)\}$ does not dominate at least two vertices in $Z \times \{3\}$.  Furthermore, each vertex of $(I - Z) \times \{1, 2, 3\}$ and $Z\times \{1\}$ is its own private neighbor. Suppose that $D - \{(z, 2)\}$ is a dominating set of $G\times K_3$ for some vertex $z \in Z_1$.  It follows that $Z_1 - \{z\}$ is a smaller subset of $Z$ that dominates $N(Z) - (\{w\} \cup N(I-Z))$. This contradicts the choice of $Z_1$. Thus, $D$ is in fact a minimal dominating set of $G\times K_3$.

Since $G\times K_3$ is well-dominated and $I \times \{1, 2, 3\}$ is also a minimal dominating set of $G\times K_3$, we have
\[|D| = 3(|I| - |Z|) + 2|Z_1| + |Z| - |Z_1| + 1 = 3|I|.\]
Therefore, $|Z_1| + 1 = 2|Z|$ or equivalently, $1 =2 |Z| - |Z_1| \ge |Z| + |Z-Z_1|$. It follows that $|Z - Z_1| = 0$ and $|Z| = 1$. However, this cannot be true since we assumed that $|Z| \ge 2$.  Therefore, no such vertex $w$ exists, and $I$ is a $2$-packing.
\end{proof}

\begin{lemma} \label{lem:onlyK3}
If $G$ is a nontrivial connected graph such that $G\times K_3$ is well-dominated, then $G = K_3$.
\end{lemma}

\begin{proof}
Suppose that $\alpha(G) \ge 2$.  For each maximum independent set $J$ of $G$, let
\[d_2(J)=\min\{d_G(a,b) \,\colon\, \{a,b\} \subseteq J \text{ and } a \neq b\}\,.\]
By Lemma~\ref{lem:2packing}, $d_2(J) \ge 3$ for every maximum independent set $J$ of $G$.  Choose a maximum independent set $I$ of $G$
such that $d_2(I) \le d_2(J)$ for every maximum independent set $J$ of $G$.  Let $u$ and $v$ be distinct vertices in $I$ such that $k=d_G(u,v)=d_2(I)$
and let $u=x_0,x_1,\ldots,x_k=v$ be a shortest $u,v$-path in $G$.  Since $I$ is a $2$-packing, $M=(I-\{u\}) \cup \{x_1\}$ is also a maximum independent
set and $d_2(M) \le d_2(I)-1$.  This contradicts the choice of $I$.  Hence $\alpha(G) = 1$, and so $G$ is a complete graph.  Using the fact that
$G \times K_3$ is also well-covered and applying Theorem~\ref{thm:2Properties} we conclude that $G=K_3$.
\end{proof}

Combining Lemma~\ref{lem:onefactork2}  and Lemma~\ref{lem:onlyK3} yields a proof of Theorem~\ref{thm:wd-direct-no-isolatable}, which gives a complete characterization of well-dominated direct products if at least one of the factors has no isolatable vertices.

\medskip

\noindent \textbf{Theorem~\ref{thm:wd-direct-no-isolatable}} \emph{
Let $G$ and $H$ be nontrivial connected graphs such that at least one of $G$ or $H$ has no isolatable vertices.  The
direct product $G \times H$ is well-dominated if and only if $G=H=K_3$ or at least one of the factors is $K_2$ and the other factor is
a $4$-cycle or the corona of a connected graph.
}

\section{Disjunctive product} \label{sec:DisjP}

In this section we will characterize well-dominated disjunctive products of connected graphs.  In particular, we prove that at least one of the factors is a complete graph and the other factor is a well-dominated graph with domination number at most $2$.  We will need several preliminary lemmas.

\begin{lemma}{\rm \cite{topp}} \label{lem:disjind}
If $I$ is a maximal independent set of $G$ and $J$ is a maximal independent set of $H$, then $I \times J$ is a maximal independent set of $G \vee H$.
\end{lemma}

\begin{lemma} \label{lem:disjwelldom}
Suppose that $G$ and $H$ have no isolated vertices.  If $A$ is any minimal total dominating set of $H$, then $\{g\} \times A$ is a minimal dominating set of $G \vee H$  for every $g \in V(G)$ that is not a universal vertex of $G$.  Similarly, if $B$ is any minimal total dominating set of $G$, then $B \times \{h\}$ is a minimal dominating set of $G\vee H$, for every $h \in V(H)$ that is not a universal vertex of $H$.
\end{lemma}
\begin{proof}
Let $A$ be a minimal total dominating set of $H$ and let $g$ be a vertex of $G$ that is not universal.  Suppose $g'\in V(G)-N_G[g]$.  Let $(v,w)$ be any vertex in $G \vee H$ that does not belong to $\{g\} \times A$.  Since $A$ is a total dominating set of $H$, there exists a vertex $a \in A$
such that $aw \in E(H)$, and it follows that $(g,a)$ is adjacent to $(v,w)$.  Hence, $\{g\} \times A$ is a dominating
set of $G \vee H$.  We claim that $\{g\} \times A$ is a minimal dominating set.  To see this let $x \in A$ and let $D=( \{g\} \times A) -\{(g,x)\}$.  There exists a vertex $h \in V(H)$ such that $h \not\in N_H(A-\{x\})$  since $A$ is a minimal total dominating set of $H$.
Thus, $D$ does not dominate $G \vee H$ since $(g',h) \not \in N[D]$.  That is, $\{g\} \times A$ is a minimal dominating  set of $G \vee H$.  The proof that $B \times \{h\}$ is a minimal dominating set of $G\vee H$, for every $h \in V(H)$ when $B$ is a minimal total dominating set of $G$ is symmetric to the above.
\end{proof}

\begin{lemma} \label{lma:condnotexist}
There does not exist a connected graph $G$ with $2 \le \alpha(G) = \gamma(G)$ and $\gamma_t(G) = 2\gamma(G)$.
\end{lemma}
\begin{proof}
Suppose  for the sake of a contradiction that $G$ is a connected graph with $2 \le m=\alpha(G) = \gamma(G)$ and $\gamma_t(G) = 2\gamma(G)$. Let $I = \{a_1, a_2, \ldots, a_m\}$ be a maximum independent set of $G$, and  for each $i\in[m]$, let $b_i$ be a specified neighbor of $a_i$.   First note that $N[a_i] \cap N[a_j] = \emptyset$ whenever $1\le i<j\le m$. Otherwise, if $u \in N[a_i] \cap N[a_j]$, then $G$ contains the  total dominating set
\[D= I \cup \{u\} \cup \bigcup_{\substack{k=1\\k \notin\{i, j\}}}^{m} \{b_k\}\,,\]
whose cardinality is less than $2\gamma(G)$.
Also, for any $i\in [m]$, if $u, v \in N[a_i]$, then $uv \in E(G)$. Otherwise, if $uv \notin E(G)$, then there exists an independent set, $\{u, v\}\cup(I - \{a_i\})$, of size $\alpha(G) + 1$. Since $G$ is connected, there exist $1\le r<s\le m$  with $w \in N(a_r)$, $y \in N(a_s)$, and $wy \in E(G)$. Reindexing if necessary, we may assume there exists $w \in N(a_1)$ and $y \in N(a_2)$ such that $wy \in E(G)$. However,
\[(I - \{a_1, a_2\}) \cup \{w, y\} \cup \bigcup_{k=3}^{m} \{b_k\}\] is a total dominating set of $G$ whose cardinality is less than $2\gamma(G)$, which is a contradiction.
\end{proof}

\begin{lemma} \label{lem:disjwelldomnotexist}
Let $G$ be a connected graph, and $H$ be a graph with no isolated vertices. If neither $G$ nor $H$
is a complete graph, then $G \vee H$ is not well-dominated.
\end{lemma}

\begin{proof}
Let $H$ be a graph with no isolated vertices such that $H$ is not a complete graph. Suppose there exists a connected graph $G$ that is not
a complete graph such that $G \vee H$ is well-dominated. Hence, $\alpha(G) \ge 2$ and $\alpha(H) \ge 2$.  By Lemma~\ref{lem:disjwelldom}, the graph $G \vee H$ has a minimal dominating set of size $\gamma_t(G)$ as well as a minimal dominating set of size $\gamma_t(H)$. In addition, by Lemma~\ref{lem:disjind}, $G \vee H$ has a minimal dominating set of size $\alpha(G)\alpha(H)$. Since $G \vee H$ is well-dominated, it must be the case that
\begin{equation} \label{eqn:disjuctive}
\alpha(G)\alpha(H)= \gamma_t(G)=\gamma_t(H).
\end{equation}

Since $\gamma(G) \leq \alpha(G)$ and $\gamma_t(G) \leq 2 \gamma(G)$, it follows from \eqref{eqn:disjuctive} that
\[2\alpha(G)\le \alpha(G)\alpha(H)= \gamma_t(G)\le 2 \gamma(G)\,,\]
which in turn implies that $\alpha(G)=\gamma(G)$.  Similarly, $\alpha(H)=\gamma(H)$.  Using~\eqref{eqn:disjuctive} again and the fact
that $2 \le \alpha(H)$ we get
\[2\gamma(G)\le \alpha(G)\alpha(H)= \gamma_t(G) \le 2\gamma(G)\,,\]
and this implies that $\gamma_t(G)=2\gamma(G)$.  Thus, $2 \le \alpha(G) = \gamma(G)$ and $\gamma_t(G) = 2\gamma(G)$. By Lemma~\ref{lma:condnotexist} such a graph does not exist, and the theorem is proved.
\end{proof}

In case one of the factors of a disjunctive product is a complete graph we have the following.  First, it is clear that
$K_1 \vee H$ is well-dominated if and only if $H$ is well-dominated.  For a disjunctive product with one of the factors
being a complete graph of order at least $2$ we can say more.

\begin{lemma} \label{lem:disjunctive-completefactor}
Let $n$ be a positive integer, $n \ge 2$.  The disjunctive product $K_n \vee H$ is well-dominated if and only if $H$ is a well-dominated graph with $\gamma(H) \leq 2$.
\end{lemma}

\begin{proof}
Let $n\ge 2$ be a positive integer.  First, suppose $K_n \vee H$ is well-dominated.   The conclusion follows if $H$ is a complete graph since the disjunctive product of two complete graphs is also complete.  Thus, we assume that $H$ has a vertex $y$ that does not dominate all of $V(H)$.  Note that if $D$ is any minimal dominating set of $H$, then for any $i \in V(K_n)$, $D' = \{i\} \times D$ is a minimal dominating set of $K_n \vee H$. It follows immediately that $|D_1|=|D_2|$ for every pair of minimal dominating sets of $H$, and therefore $H$ is well-dominated with
$\gamma(H)=\gamma(K_n \vee H)$.  Furthermore, for any two distinct vertices  $i$ and $j$ of $K_n$, the set $\{(i, y), (j, y)\}$ is a minimal, and hence a minimum, dominating set of $K_n \vee H$.  We conclude that $\gamma(H)=2$.

Now, suppose $H$ is a well-dominated graph with $\gamma(H) \leq 2$. By Proposition~\ref{prop:wd-Implies-wc}, $H$ is also well-covered and thus  $\alpha(H)=\gamma(H)$. If $\alpha(H) = 1$, then $H$ is a complete graph and $K_n \vee  H$ is a complete graph and thus is well-dominated.  Next, assume $\alpha(H) =\gamma(H)= 2$.  This implies that $K_n \vee H$ has no universal vertex.  It follows that any subset of $V(K_n \vee H)$ consisting of two vertices from distinct $H$-layers is a minimal dominating set.  If $D'$ is a minimal dominating set of $K_n \vee H$ such that $|D'| \ge 3$, then $D'=\{i\} \times D$ for some $i \in [n]$ and some minimal dominating set $D$ of $H$.  Since $H$ is well-dominated with $\gamma(H)=2$, there is no such set $D'$.  Therefore, all minimal dominating sets of $K_n \vee H$ have cardinality $2$, which implies that $K_n \vee H$ is well-dominated.
\end{proof}

Combining Lemmas~\ref{lem:disjwelldomnotexist} and~\ref{lem:disjunctive-completefactor} we get the promised characterization of well-dominated disjunctive products.

\noindent \textbf{Theorem~\ref{thm:disj-characterization}} \emph{
Let $G$ and $H$ be nontrivial connected graphs.  The disjunctive product $G \vee H$ is well-dominated if and only if at least one of $G$ or $H$ is a complete graph and the other factor is well-dominated with domination number at most $2$.
}


\begin{thebibliography}{99}

\bibitem{bc-1992} B.~Bollob\'{a}s and E.~Cockayne.  Graph theoretic parameters concerning domination, independence and irredundance.
   \textit{J. Graph Theory}, {\bf 3}: 241--250 (1979)

\bibitem{domdirect}  B.~Bre\v{s}ar, S.~Klav\v{z}ar and D.~F.~Rall.  Dominating direct products of graphs.  \textit{Discrete Math.}, {\bf 307}(13): 1636--1642 (2007)

\bibitem{girth5} A.~Finbow, B.~Hartnell, and R.~Nowakowski. A characterization of well-covered graphs of girth $5$ or greater. \textit{J. Combin. Theory Ser. B}, {\bf 57}: 44--68 (1993)

\bibitem{wd} A.~Finbow, B.~Hartnell, and R.~Nowakowski. Well-dominated graphs: a collection of well-covered ones. \textit{Ars Combin.}, {\bf 25-A}: 5--10 (1988)

\bibitem{gj-1979} M.~R.~Garey and D.~S.~Johnson. \textit{Computers and Intractability: A Guide to the Theory of NP-completeness}, W.H. Freeman \& Co., New York, NY, USA, (1979)

\bibitem{gks-2011} T.~J.~Gionet Jr., E.~L.~C.~King and Y.~Sha.   A revision and extension of results on $4$-regular, $4$-connected, claw-free graphs.  \textit{Discrete Appl. Math.}, {\bf 159}(12): 1225--1230  (2011)


\bibitem{ghm-2017} D.~G\"{o}z\"{u}pek, A.~Hujdurovi\'{c} and M.~Milani\v{c}. Characterizations of minimal dominating sets and the well-dominated property in lexicographic product graphs. \textit{Discrete Math. Theor. Comput. Sci.}, {\bf 19}(1), Paper No. 25, 17 pp. (2017)

\bibitem{ha-ra-2013} B.~Hartnell and D.~F.~Rall. On the Cartesian product of non well-covered graphs. \textit{Electron.~J.~Combin.}, {\bf 20}(2) \#P21: 1--4 (2013)

\bibitem{wellcoverCart} B.~Hartnell, D.~F.~Rall, and K.~Wash. On well-covered Cartesian products. \textit{Graphs and Combin.}, {\bf 34}(6): 1259--1268 (2018)


\bibitem{KKDFR} K.~Kuenzel and D.~F.~Rall.  On well-covered direct products.   submitted (2019)

\bibitem{lt-2017} V.~E.~Levit and D.~Tankus. Well-dominated graphs without cycles of lengths 4 and 5. \textit{Discrete Math.}, {\bf 340} (80): 1793--1801 (2017)

\bibitem{pa-xu-1982} C.~Payan and N.~H.~Xuong.  Domination-balanced graphs.  \textit{J. Graph Theory}, {\bf 6}: 23--32 (1982)

\bibitem{p-1970} M.~Plummer. Some covering concepts in grahs. \textit{J. Combinatorial Theory}, {\bf 8}: 253--287 (1970)

\bibitem{tv-1990} J.~Topp and L.~Volkmann. Well covered and well dominated block graphs and unicyclic graphs. \textit{Math. Pannon.}, {\bf 1}(2): 55--66  (1990)

\bibitem{topp} J.~Topp and L.~Volkmann. On the well-coveredness of products of graphs. \textit{Ars Combin.}, {\bf 33}: 199--215 (1992)

\end{thebibliography}
\end{document}